\documentclass{amsart}[12pt]
\usepackage{amsmath, amsthm, amscd, amsfonts,}
\usepackage{graphicx,float}

\usepackage{amssymb}

\usepackage{pb-diagram}

\DeclareMathOperator{\cf}{cf}

\DeclareMathOperator{\ran}{ran}

\def\MPB{{\mathbb{P}}}

\def\MCB{{\mathbb{C}}}

\def\k{\kappa}
\def\sse{\subseteq}

\def\l{\lambda}
\def\lan{\langle}
\def\ran{\rangle}
\def\a{\alpha}
\def\om{\omega}

\def\ov{\overline}

\newtheorem{theorem}{Theorem}[section]
\newtheorem{lemma}[theorem]{Lemma}

\newtheorem{corollary}[theorem]{Corollary}

\newtheorem{remark}[theorem]{Remark}
\newtheorem{notation}[theorem]{Notation}
\newtheorem{claim}[theorem]{Claim}
\numberwithin{equation}{section}

\def\MPB{{\mathbb{P}}}

\def\MCB{{\mathbb{C}}}

\def\k{\kappa}
\def\sse{\subseteq}

\def\l{\lambda}
\def\lan{\langle}
\def\ran{\rangle}
\def\a{\alpha}
\def\om{\omega}

\def\ov{\overline}

\def\l{\lambda}

\def\sse{\subseteq}

\def\lan{\langle}
\def\ran{\rangle}
\def\ov{\bar}
\def\rmark{\mbox{$\rm\bf\rule{0.06em}{1.45ex}\kern-0.05em R$}}
\def\pmark{\mbox{$\rm\bf\rule{0.06em}{1.45ex}\kern-0.05em P$}}
\def\nmark{\mbox{$\rm\bf\rule{0.06em}{1.45ex}\kern-0.05em N$}}
\def\vdash{\mbox{$\rm\| \kern-0.13em -$}}
\newcommand{\lusim}[1]{\smash{\underset{\raisebox{1.2pt}[0cm][0cm]{$\sim$}}
{{#1}}}}

\usepackage{pb-diagram}

\def\l{\lambda}

\def\sse{\subseteq}

\def\lan{\langle}
\def\ran{\rangle}
\def\ov{\bar}
\def\rmark{\mbox{$\rm\bf\rule{0.06em}{1.45ex}\kern-0.05em R$}}
\def\pmark{\mbox{$\rm\bf\rule{0.06em}{1.45ex}\kern-0.05em P$}}
\def\nmark{\mbox{$\rm\bf\rule{0.06em}{1.45ex}\kern-0.05em N$}}
\def\vdash{\mbox{$\rm\| \kern-0.13em -$}}

\begin{document}

\title[Adding a lot of Cohen reals by adding a few II]{Adding a lot of Cohen reals by adding a few II }

\author[M. Gitik and M. Golshani]{Moti Gitik and Mohammad Golshani}

\thanks{} \maketitle




\begin{abstract}
We study pairs $(V, V_{1})$,  $V \subseteq V_1$,  of models of $ZFC$  such that adding $\kappa-$many Cohen reals over $V_{1}$ adds $\lambda-$many Cohen reals over $V$ for some $\lambda> \kappa$.
\end{abstract}
\maketitle

\section{Introduction}
We continue our study from [3].  We study pairs $(V, V_{1})$,  $V \subseteq V_1$,  of models of $ZFC$ with the same ordinals, such that adding $\kappa-$many Cohen reals over $V_{1}$ adds $\lambda-$many Cohen reals over $V$ for some $\lambda> \kappa$\footnote{By ``$\lambda-$many Cohen reals'' we mean "a generic object $\langle s_\a : \a < \l\rangle$ for the poset $\mathbb{C}(\l)$ of finite partial
functions from $\l\times\omega$ to $2$''.}. We are mainly interested when $V$ and $V_{1}$ have the same cardinals and reals. We prove that for such models, adding $\kappa-$many Cohen reals over $V_{1}$ cannot produce more Cohen reals over $V$ for $\kappa$ below the first fixed point of the $\aleph-$function, but the situation at the first fixed point of the $\aleph-$function is different. We also reduce the large cardinal assumptions from [1, 3] to the optimal ones.

\section{Adding many Cohen reals by adding a few: a general result }

In this section we prove the following general result.
\begin{theorem}
Suppose $\kappa < \lambda$ are infinite (regular or singular) cardinals, and let $V_1$
be an  extension of $V.$ Suppose that in $V_1:$

$(a)$ $\kappa < \lambda$ are still infinite cardinals\footnote{$\lambda$ can be a regular or a singular cardinal, but by $(b)$, $\k$ is necessarily a singular cardinal of cofinality $\omega$.},

$(b)$ there exists an increasing sequence $\langle \kappa_n: n < \omega  \rangle$ of regular cardinals, cofinal in $\kappa.$ In particular $cf(\kappa) = \omega,$

$(c)$ there is an increasing (mod finite) sequence $\langle
f_{\alpha}: \alpha < \lambda \rangle$ of functions in the product

$\hspace{.5cm}$$\prod_{n< \omega}(\kappa_{n+1}\setminus\kappa_n),$

$(d)$ there is a splitting $\langle S_{\sigma} : \sigma<\kappa \rangle$ of $\lambda$ into sets of size $\lambda$ such
that for every countable

$\hspace{.5cm}$set $I \in V$ and  every $\sigma<\kappa$ we have $|I\cap S_{\sigma}|< \aleph_0.$

Then adding $\kappa-$many Cohen reals over $V_1$ produces $\lambda-$many Cohen reals over $V.$
\end{theorem}
\begin{remark}
Condition $(c)$ holds automatically for $\lambda=\kappa^+$; given any collection $\mathcal{F}$ of $\kappa$-many elements of $\prod_{n<\omega}(\kappa_{n+1}\setminus\kappa_n),$ there exists $f$ such that for each $g \in \mathcal{F}, f(n)> g(n)$ for all large $n$ \footnote{ To see this let $\mathcal{F}=\bigcup_{n<\omega}\mathcal{F}_n,$ where $\mathcal{F}_0 \subseteq \mathcal{F}_1 \subseteq ...$ and $|\mathcal{F}_n|< \k_{n+1},$ and define $f$ so that $\sup\{g(n): g\in \mathcal{F}_n    \} < f(n)\in \k_{n+1}\setminus \k_n$.}. Thus we can define by induction on $\alpha< \kappa^+$, an increasing (mod finite) sequence $\langle
f_{\alpha}: \alpha < \kappa^+ \rangle$ in $\prod_{n<\omega}(\kappa_{n+1}\setminus\kappa_n)$\footnote{ Let $f_0$ be arbitrary. Given $\a<\k^+,$ we can apply the above to find $f_\a$ so that $f_\a(n) > f_\beta(n),$ for all large $n,$ and all $\beta<\a$.}.
\end{remark}

\begin{proof}
Force to add
$\k-$many Cohen reals over $V_1$. Split them into $\lan r_{i,\sigma} : i,\sigma<\k\ran$ and $\lan r'_{\sigma} : \sigma<\k\ran$.
Also in $V,$ split  $\kappa$  into $\kappa-$blocks $B_{\sigma}, \sigma<\kappa,$ each of size $\kappa,$ and  let $\lan f_\a : \a<\l\ran\in V_{1}$ be an
increasing (mod finite) sequence in $\prod_{n<\om}(\k_{n+1}
\setminus \k_{n})$. Let $\a<\l$. We define a real $s_\a$ as
follows. Pick $\sigma<\kappa$ such that $\a\in S_{\sigma}.$ Let $k_{\a}=min\{k<\omega: r'_{\sigma}(k) \}=1$ and set

\begin{center}
$\forall n<\om$, $s_\a(n)=r_{f_\a(n+k_{\a}),\sigma}(0)$.
\end{center}

The following lemma completes the proof.

\begin{lemma} $\lan s_\a:\a<\l\ran$ is a sequence of
$\l-$many Cohen reals over $V$.
\end{lemma}
\begin{notation} $(a)$ For a forcing notion $\mathbb{P}$ and $p,q\in \mathbb{P},$ we let $p\leq q$  mean $p$ is stronger than $q$.

$(b)$ For each set $I$, let ${\MCB}(I)$ be the Cohen
forcing notion for adding $I-$many Cohen reals. Thus
${\MCB}(I)=\{p:p$ is a finite partial function from $I\times \om$
into 2 $\}$, ordered by $p\leq q$ iff $p \supseteq q$.
\end{notation}
\begin{proof} First note that $\lan\lan r_{i,\sigma} :
i,\sigma<\k\ran, \lan r'_{\sigma} : \sigma<\k\ran\ran$ is ${\MCB}(\k\times\k)\times \MCB(\k)-$generic over $V_1$. By the $c.c.c.$ of ${\MCB}(\l)$ it
suffices to show that for  any countable set $I\sse \l$, $I\in V$,
the sequence $\lan s_\a:\a\in I\ran$ is ${\MCB}(I)-$generic over
$V$. Thus it suffices to prove the following

$\hspace{1.5cm}$ For every $(p,q)\in {\MCB}(\k\times\k)\times\MCB(\k)$
and every open dense subset $D\in V$

(*)$\hspace{1cm}$ ~~of ${\MCB}(I)$,~ there is
$(\ov{p}, \ov{q})\leq (p,q)$ such~ that $(\ov{p},\ov{q}) \vdash
\ulcorner   \lan \lusim{s}_\a : \a\in I\ran$ extends

$\hspace{1.5cm}$  some element of $D\urcorner. $

 Let $(p, q)$ and $D$ be as above and for simplicity suppose that
$p=q=\emptyset.$ Let $b\in D,$ and let $\a_1, ..., \a_m$ be an enumeration of the components of $b,$ i.e., those $\a$ such that  $(\a,n)\in dom(b)$ for some $n.$ Also let $\sigma_1, ..., \sigma_m<\k$ be such that $\a_i \in S_{\sigma_i}, i=1, ..., m.$ By $(d)$ each $I\cap S_{\sigma_i}$ is finite, thus by $(c)$ we can find $n^*<\omega$ such that for all $n\geq n^*, 1\leq i \leq m$ and $\a^*_1<\a^*_2$ in $I\cap S_{\sigma_i}$ we have
$f_{\a^*_1}(n)<f_{\a^*_2}(n).$ Let
\begin{center}
$\ov{q}=\{\langle \sigma_i, n, 0 \rangle: 1\leq i \leq m, n< n^*     \}.$
\end{center}
Then $\ov{q}\in \MCB(\k)$ and $(\emptyset, \ov{q})\vdash \ulcorner k_{\a_i}\geq n^*  \urcorner$ for all $1\leq i \leq m.$ Let
\begin{center}
$\ov{p}=\{ \langle f_{\a_i}(n+k_{\a_i}), \sigma_i, 0, b(\a_i, n) \rangle: 1\leq i\leq m, (\a_i, n)\in dom(b) \}.$
\end{center}
Then $\ov{p}\in \MCB(\k\times\k)$ is well-defined and for $(\a_i,n)\in dom(b), 1\leq i \leq m$ we have
\begin{center}
$(\ov{p},\ov{q})\vdash \ulcorner \lusim{s}_{\a_i}(n)= \lusim{r}_{f_{\a_i}(n+k_{\a_i}),\sigma_i}(0)=\ov{p}(f_{\a_i}(n+k_{\a_i}), \sigma_i, 0)= b(\a_i, n)   \urcorner$
\end{center}
and hence
\begin{center}
$(\ov{p},\ov{q})\vdash \ulcorner \langle \lusim{s}_{\a}: \a \in I   \rangle$ extends $b \urcorner.$
\end{center}
(*) follows and we are done.
\end{proof}
The theorem follows.
\end{proof}

\section{Getting results from optimal hypotheses}

\begin{theorem}
Suppose $GCH$ holds and $\kappa$ is a cardinal of countable cofinality and
there are $\kappa-$many measurable cardinals below $\kappa.$ Then
there is a cardinal preserving not adding a real extension $V_1$
of $V$ in which there is a splitting $\langle S_\sigma: \sigma<\k    \rangle$ of $\k^+$ into sets of size $\k^+$ such that for every countable set $I\in V$ and every $\sigma<\k, |I\cap S_\sigma|<\aleph_0.$

\end{theorem}
\begin{proof}
Let $X$ be a set of measurable cardinals below $\kappa$ of size
$\kappa$ which is discrete, i.e., contains none of its limit
points, and for each $\xi \in X$ fix a normal measure $U_{\xi}$ on
$\xi.$ For each $\xi \in X$ let  $\MPB_{\xi}$ be the Prikry
forcing associated with the measure $U_{\xi}$ and let $\MPB_{X}$
be the Magidor iteration of $\MPB_{\xi}$'s, $\xi\in X$ (cf. [2, 5]).
 Since $X$ is discrete, each condition in $\MPB_{X}$ can be seen as $p=\langle\langle
s_{\xi}, A_{\xi} \rangle: \xi \in X \rangle$ where for
$\xi \in X, \langle  s_{\xi}, A_{\xi} \rangle \in
\MPB_{\xi}$ and $supp(p)=\{\xi \in X: s_{\xi}\neq \emptyset
\}$ is finite. We may further suppose that for each $\xi \in X$
the Prikry sequence for $\xi$ is contained in $(sup(X \cap \xi),
\xi).$ Let $G$ be $\MPB_{X}-$generic over $V$. Note that $G$ is
uniquely determined by a sequence $(x_{\xi}: \xi \in X)$, where
each $x_{\xi}$ is an $\omega-$sequence cofinal in $\xi,$ $V$ and $V[G]$ have the same cardinals, and  $GCH$ holds in $V[G].$

Work in $V[G]$. We now force  $\langle S_\sigma: \sigma<\k\rangle$ as
follows. The set of conditions $\mathbb{P}$ consists of pairs
$p=\langle \tau, \lan s_\sigma: \sigma<\k \ran \rangle\in V[G]$ such that:

$(1)$ $\tau<\k^+,$

$(2)$ $\lan s_\sigma: \sigma<\k \ran$ is a splitting of $\tau$,

$(3)$ for every countable set $I\in V$ and every $\sigma<\k, |I\cap s_\sigma|<\aleph_0.$

\begin{remark}
$(a)$ Given a condition $p\in \mathbb{P}$ as above, $p$ decides an initial segment of $S_\sigma,$ namely $S_\sigma\cap \tau$, to be $s_\sigma$.  Condition $(3)$ guarantees that each component in this initial segment has finite intersection with countable sets from the ground model.

$(b)$ Let $t_0=\bigcup_{\xi\in X}x_\xi.$ By genericity arguments, it is easily seen that $t_0$ is a subset of $\kappa$ of size $\kappa$ such that for all countable sets $I\in V, |I\cap t_0|<\aleph_0$. For each $i<\kappa$ set $t_i=t_0+i=\{\alpha+i: \alpha\in t_0 \}.$ Then again by genericity arguments, for every countable set $I\in V, |I\cap t_i|<\aleph_0$. Define $s_i, i<\kappa$ by recursion as $s_0=t_0$ and $s_i=t_i \setminus \bigcup_{j<i}t_j$ for $i>0.$ Then $p=\langle \kappa, \lan s_\sigma: \sigma<\k \ran \rangle\in \mathbb{P}$ (since again by genericity arguments, $\lan s_\sigma: \sigma<\k \ran$ is a splitting of $\k$), and hence $\mathbb{P}$ is non-trivial.
\end{remark}

We call $\tau$ the height of $p$ and denote it by $ht(p).$
For $p=\langle \tau, \lan s_\sigma: \sigma<\k \ran \rangle$ and $q=\langle \nu, \lan t_\sigma: \sigma<\k \ran \rangle$ in $\MPB$ we define $p\leq q$  iff

$(1)$ $\tau\geq \nu,$

$(2)$ for every $\sigma<\k, s_{\sigma} \cap\nu=t_\sigma,$ i.e., each $s_\sigma$ end extends $t_\sigma.$

\begin{lemma}
$(a)$ $\mathbb{P}$ satisfies the $\kappa^{++}-c.c,$

$(b)$ $\mathbb{P}$ is $<\kappa-$distributive.
\end{lemma}
\begin{proof}
$(a)$ is trivial, as $|\mathbb{P}| \leq 2^\k=\k^+$. For $(b)$, fix
$\delta < \kappa, \delta$ regular, and let $p \in \mathbb{P}$
and $\lusim{g} \in V[G]^{\mathbb{P}}$ be such that
\begin{center}
$p \vdash \ulcorner \lusim{g}: \delta \rightarrow On  \urcorner.$
\end{center}
We find $q \leq p$ which decides $\lusim{g}.$
Fix in $V$ a splitting of $\kappa$ into $\delta-$many sets of size $\kappa, \langle Z_i :i<\delta \ran$ \footnote{ Note that this is possible, as $\delta<\kappa$ are cardinals in $V$. The splitting can also be chosen in $V[G].$}.
Let $\theta$ be a large enough regular cardinal. Pick an increasing continuous sequence $\langle M_{i}: i \leq \delta  \rangle$ of elementary submodels of $\langle H(\theta), \in \rangle$ of size $\kappa$ such that \footnote{Condition $(5)$ can be guaranteed using the fact that the set $K=\{\a<\k^+: cf^V(\a)\in X \}$ is a stationary subset of $\k^+$ in $V[G]$ (given $M_i$, build a suitable continuous increasing chain $\langle N_j: j<\k^+  \rangle$ consisting of models of size $\k.$ Then $\langle \sup(N_j\cap \k^+): j<\k^+  \rangle$ forms a club of $\k^+,$ and $M_{i+1}$ can be chosen to be one of those $N_j$ so that $\sup(N_j\cap \k^+)\in K$). $(7)$ can be guaranteed by the fact that $\mathbb{P}_X$ satisfies the $\k^+$-$c.c.$ and the models have size $\k$ (use the fact that given any model $N$ of size $\k,$ there exists a model in $V$ of the same size which contains $N\cap V$).}:
\begin{enumerate}
\item $\langle M_{i}: i \leq \delta  \rangle\in V[G],$
\item $p, \mathbb{P}, \lusim{g}, \langle Z_i : i<\delta \ran \in M_{0},$
\item if $i < \delta$ is a limit ordinal, then $\langle M_j: j\leq i\rangle \in M_{i+1},$
\item $cf(M_{\delta} \cap \kappa^{+})=\delta,$
\item if $i$ is not a limit ordinal, then $cf^{V}(M_{i+1} \cap \kappa^{+})= \xi_{i}$ for a measurable $\xi_{i}$ of $V$ in $X$,
\item $i<j \Rightarrow \xi_{i} < \xi_{j},$
\item $\langle M_{i} \cap V: i \leq \delta \rangle \in V.$
\end{enumerate}
For each non-limit $i< \delta, M_{i+1}\cap V$ is in $V$ by clause $(7),$ and so by clause $(5), cf^{V}(M_{i+1} \cap \kappa^{+})= \xi_{i}$, where $\xi_i\in X,$ so we can pick in $V$ a cofinal in $M_{i+1} \cap \kappa^{+}$ sequence   $\langle \eta^{i}_{\a}: \alpha < \xi_{i} \rangle,$ where $\eta^i_\a > M_i\cap \k^+,$ for all $\a<\xi_i$ \footnote{Note that $\sup(M_{i+1} \cap \kappa^{+})=M_{i+1} \cap \kappa^{+}.$ This is because if $\xi<\k^+,$ and $\xi\in M_{i+1},$ then since $\k\cup\{\k\} \subseteq M_{i+1},$ and $M_{i+1}\models |\xi|=\k,$ we have $\xi \subseteq M_{i+1}$. Also, as the sequence of $M_i$'s in increasing continuous, $\sup(M_{i} \cap \kappa^{+})=M_{i} \cap \kappa^{+}$ holds for limit ordinals $i$.}.

Denote by $\xi_i^{'}$ the first element of the Prikry sequence of
$\xi_{i}.$ We define a descending sequence $p_i=\langle \tau_i, \lan s_{i,\sigma}: \sigma<\k \ran \rangle$ of conditions by induction as
follows:

{\bf i=0.} Set $p_{0}=p.$

{\bf i=j+1.} Assume $p_{j}$ is constructed such that $p_{j} \in
M_{j}$  if $j$ is not a limit ordinal, and  $p_{j} \in
M_{j+1}$ if $j$ is a limit ordinal  and $p_{j}$ decides $\lusim{g}\upharpoonright j.$ Fix a bijection  $f_{j}:Z_j \rightarrow (ht(p_j),\eta_{\xi'_{j}}^{j})$ in $M_{j+1}$  and set \footnote{It is easily seen by induction on $j\leq i$ that $ht(p_j) < \eta^j_{\xi'_j}$:
 if $j=0$ or $j$ is a successor ordinal, then $p_{j}\in M_{j},$ so
  $ht(p_j) \in M_j\cap \k^+ < \eta^j_{\xi'_j}$. If $j$ is a limit ordinal, then  $ht(p_j)=\sup_{k<j}ht(p_k) \leq \sup_{k<j}M_k\cap \k^+=M_j\cap \k^+ < \eta^j_{\xi'_j}.$}
\begin{center}
$p'_{j+1}=\langle \eta_{\xi'_{j}}^{j},  \langle s_{j,\sigma}\cup \{f_{j}(\sigma)\}:\sigma\in Z_j \rangle^{\frown}\lan s_{j,\sigma}: \sigma\in \k\setminus Z_j  \ran    \rangle$
\end{center}
 Clearly $p'_{j+1} \in M_{j+1}.$ Let $p_{j+1} \in M_{j+1}$ be an
extension of $p_{j+1}^{'}$ which decides $\lusim{g}(j).$

{\bf limit(i).} Let $p_{i}= \langle sup_{j<i}ht(p_j), \lan \bigcup_{j<i}s_{j,\sigma}:\sigma<\k   \ran \rangle.$

Let us show that the above sequence is well-defined. Thus we need
to show that for each $i \leq \delta, p_{i} \in \mathbb{P}.$ We
prove this by induction on $i$. The successor case is trivial.
Thus fix a limit ordinal $i \leq \delta$. If $p_{i} \notin
\mathbb{P},$ we can find a countable set $I \in V, I \subseteq \k^+,$ and $\sigma<\k$ such that $I\cap s_{i,\sigma}$ is infinite.
 Define the sequence
$\langle \alpha(j): j<i \rangle$ as follows:
\begin{itemize}
\item if $I \cap (M_{j+1}\setminus M_{j}) \neq \emptyset,$ then $\alpha(j)
\in [\sup(X\cap\xi_j), \xi_j]$ is the least such that
$\eta_{\alpha(j)}^{j}> sup(I \cap (M_{j+1}\setminus M_{j})),$ \item
$\alpha(j)=\sup(X\cap \xi_j)$ otherwise. Note that in this case $\a(j) <\xi'_j$ (because the Prikry sequence for $\xi$ was chosen in the interval $(sup(X\cap \xi), \xi)$).
\end{itemize}
Clearly $\langle \alpha(j): j<i  \rangle \in V.$
\begin{lemma}
The set $K=\{j<i: \xi_{j}^{'} \leq \alpha(j) \}$ is finite.
\end{lemma}
\begin{proof}
Let $p \in \MPB_{X}, p=\langle\langle
s_{\xi}, A_{\xi} \rangle: \xi \in X \rangle.$
 Extend $p$ to $q=\langle\langle
t_{\xi}, B_{\xi} \rangle: \xi \in X \rangle$ by setting
\begin{itemize}
\item $t_{\xi}=s_{\xi}$ and $B_{\xi}=A_{\xi}$ for $\xi
\in supp(p),$ \item $t_{\xi}=\emptyset$ and $B_{\xi}=
A_{\xi} \backslash (\a(j)+1),$ if $\xi=\xi_j$ (some $j<i$) and $\xi\notin supp(p)$,
\item $t_{\xi}=\emptyset$ and $B_{\xi}=
A_{\xi}$, otherwise.
\end{itemize}
Then $q \leq p$ and $q \vdash \ulcorner  \lusim{K} \subseteq \{j<i: \xi_j\in supp(p)    \} \urcorner$,  so $q \vdash \ulcorner \lusim{K}$ is finite $ \urcorner.$
\end{proof}
Take $i_0<i$ large enough so that no point $\geq i_0$ is in $K.$ Then for all $j\geq i_0$ we have $\xi_{j}^{'} > \alpha(j),$ hence
$\eta_{\xi_{j}^{'}}^{j}> \sup(I \cap (M_{j+1}))$ \footnote{This is trivial if  $I \cap (M_{j+1}\setminus M_{j}) \neq \emptyset,$ as then $\eta_{\xi_{j}^{'}}^{j} > \eta_{\a(j)}^{j} > \sup(I \cap (M_{j+1}\setminus M_{j}))=\sup(I\cap M_{j+1}).$ If  $I \cap (M_{j+1}\setminus M_{j})= \emptyset,$ then $\eta_{\xi_{j}^{'}}^{j} > M_{j}\cap \k^+=\sup(M_j\cap \k^+) \geq \sup(I\cap M_j)$ (as $I \subseteq \k^+$) and $\sup(I\cap M_{j+1})=\sup(I\cap M_j)$ (since $I$ has no points in $M_{j+1}\setminus M_{j}$), and hence again $\eta_{\xi_{j}^{'}}^{j}> \sup(I \cap (M_{j+1}))$.}.

\begin{claim}
We have
\begin{center}
$I\cap s_{i,\sigma}\subseteq  I \cap (s_{i_0,\sigma}\cup \{f_{i_1}(\sigma) \})$
\end{center}
where $i_1$ is the unique ordinal less than $\delta$ so that $\sigma\in Z_{i_1}.$
\end{claim}
\begin{proof}
Assume towards a contradiction that the inclusion fails, and let $t\in I\cap s_{i, \sigma}$ be such that $t\notin  I \cap (s_{i_0,\sigma}\cup \{f_{i_1}(\sigma) \}).$ As $i$ is a limit ordinal, $I\cap s_{i,\sigma}=I \cap \bigcup_{j<i}s_{j,\sigma}$. Let $j<i$ be the least such that $t\in s_{j+1, \sigma}.$ Then as $t\in I\cap M_{j+1}$ and $j\geq i_0$ we have $t<\eta^j_{\xi'_j},$ so that by our definition of $p'_{j+1}, t$ must be of the form $f_j(\sigma),$ where $\sigma\in Z_j.$ But then $j=i_1$ and hence $t=f_{i_1}(\sigma).$ This is a contradiction, and the result follows.
\end{proof}
Thus, as $I\cap s_{i,\sigma}$ is infinite, we must have $I\cap s_{i_0,\sigma}$ is also infinite, and this is in
contradiction with our inductive assumption.

It then follows that $q=p_{\delta} \in \mathbb{P}$ and it decides $\lusim{g}.$
\end{proof}
Let $H$ be $\mathbb{P}-$generic over $V[G]$ and set
$V_{1}=V[G][H].$ It follows from Lemma 3.3 that all cardinals $\leq \k$ and $\geq \k^{++}$ are preserved. Also note that $\k^+$ is preserved, as otherwise it would have cofinality less that $\kappa$, which is impossible by the $<\k-$distributivity of $\mathbb{P}$. Hence $V_{1}$ is a cardinal preserving
and not adding reals forcing extension of $V[G]$ and hence of $V$.
For $\sigma<\k$ set $S_{\sigma}=\bigcup_{\langle \tau, \lan s_\sigma: \sigma<\k \ran \rangle\in H}s_{\sigma}.$

\begin{lemma}
The sequence $\lan S_\sigma: \sigma<\k \ran$ is as required.
\end{lemma}
\begin{proof}
For each $\tau<\kappa^+,$ it is easily seen that the set of all conditions $p$ such that $ht(p)\geq \tau$ is dense, so $\lan S_\sigma: \sigma<\k \ran$ is a partition of $\k^+.$ Now suppose that $I\in V$ is a countable subset of $\k^+.$ Find $p=\langle \tau, \lan s_\sigma: \sigma<\k \ran \rangle\in H$ such that $\tau \supseteq I.$ Then for all $\sigma<\k, S_\sigma\cap I=s_\sigma\cap I,$ and hence $|S_\sigma\cap I|=|s_\sigma\cap I|<\aleph_0.$
\end{proof}

Theorem 3.1 follows.
\end{proof}

\begin{remark}
$(a)$The size of a set $I$ in $V$ can be changed from countable to any fixed $\eta < \kappa.$
Given such $\eta,$ we start with the Magidor iteration of Prikry forcings above $\eta$.\footnote{The reason for starting the iteration above $\eta$ is to add no subsets of $\eta$. This will guarantee that if $t_0$ is defined as in Remark 3.2$(b)$, then $t_0$ has finite intersection with sets from $V$ of size $\eta$. Using this fact we can show as before that there is a splitting of $\k$ into $\k$ sets, each of them having finite intersection with ground model sets of size $\eta$. This makes the second step of the above forcing construction well-behaved.} The rest of the conclusions are the same.

$(b)$ It is possible to add a one element Prikry sequence to each $\xi \in X.$\footnote{Conditions in the forcing are of the form $\langle p_\xi: \xi\in X \rangle,$ where for each $\xi\in X, p_\xi$ is either of the form $A_\xi$ for some $A_\xi\in U_\xi$, or $\alpha_\xi$ for some $\alpha_\xi <\xi.$ We also require that there are only finitely many $p_\xi$'s of the form $\alpha_\xi$. When extending a condition, we allow either $A_\xi$ to become thinner, or replace it by some ordinal $\alpha_\xi \in A_\xi$.} Then $V_1$ will be a cofinality preserving generic extension of $V$.
\end{remark}

The next corollary follows from Theorem 3.1 and Remark 2.2.
\begin{corollary}
Suppose that $GCH$ holds in $V,$ , $\kappa$ is a cardinal of countable cofinality and
there are $\kappa-$many measurable cardinals below $\kappa$.
Then there is a cardinal preserving not adding a real extension
$V_1$ of $V$ such that adding $\k-$many Cohen reals over
$V_1$ produces $\k^{+}-$many Cohen reals over $V$.
\end{corollary}

\begin{theorem}Assume that there is no sharp for a strong cardinal.
Suppose $V_1\subseteq V_2$  have the same cardinals, same reals and
there is an infinite set of ordinals $S$ in $V_2$
which does not contain an infinite subset which is in $V_1$.
Then either
\begin{enumerate}
\item $S$ is countable, and then there is a measurable cardinal  $\leq\sup(S)$ in $\mathcal{K}$,
\\or
\item $S$ is uncountable, and then
there is $\delta\leq\sup(S)$ which is a limit of $|S|$--many or $\delta-$many measurable cardinals of $\mathcal{K}$.
\end{enumerate}
\end{theorem}
\begin{proof}
Given a model $V$, let $\mathcal{K}(V)$ denote the core model of $V$ below the strong cardinal.
Note that $\mathcal{K}(V_1)=\mathcal{K}(V_2),$  since the models $V_1$ and $V_2$ agree about cardinals. We denote this common core model by $\mathcal{K}.$

Let us first assume that $S$ is countable.
Suppose otherwise, i.e., there are no measurable cardinals $\leq\sup(S)$ in $\mathcal{K}$. Then by the Covering Theorem (see [6]) there is $Y \in \mathcal{K}$, $|Y|=\aleph_1$ which covers $S$.
Fix some $f:\aleph_1 \leftrightarrow Y$ in $V_1$. Consider $Z = f^{-1''}S$.
Then $Z$ also does not contain an infinite subset which is in $V_1$.
But $Z$ is countable, hence there is $\eta<\omega_1$ with $Z \subseteq \eta$.
Let  $g:\omega \leftrightarrow \eta$ in $V_1$.
Consider $X = g^{-1''}Z$. Then $X$ also does not contain an infinite subset which is in $V_1$.
But this is impossible since $V_1,V_2$ have the same reals (and hence $X$ itself is in $V_1$).
Contradiction.

Let us deal now with the uncountable case.
Suppose otherwise, i.e., there is no $\delta\leq\sup(S)$ which is a limit of $|S|-$many or $\delta-$many measurable cardinals of $\mathcal{K}$. Pick a counterexample $S$ with $\sup(S)$ as small as possible.
Denote $\sup(S)$ by $\delta$.
By minimality, $\delta$ is a cardinal.
Also, the measurable cardinals of $\mathcal{K}$ are unbounded in $\delta$.
For otherwise, let $\xi$ be their  supremum.
Pick $S' \subseteq S$ of size $\xi$.
By the Covering Theorem, $S'$ can be covered by a set in $\mathcal{K}$ of size $\xi<\delta$,
and then we get a contradiction to the minimality of $\delta$, as witnessed by $\xi$ and $S'$ \footnote{We then have $\sup(S')=\xi < \delta$ and $S'$ is a counterexample to our assumption of smaller supremum.}.

Clearly,  $\delta$ must be a singular cardinal and by the above, $\delta$ is a limit of measurable cardinals  in  $\mathcal{K}$. Fix a cofinal sequence $\langle \delta_i : i<\cf(\delta) \rangle$.
Denote by $\eta $ the cardinality of the set $\{\alpha<\delta : \alpha \text{ is a measurable cardinal in  } \mathcal{K} \}$.
By the assumption, $|S|>\eta\geq \cf(\delta)$. But then there is $i^*<\cf(\delta) $ such that $S \cap \delta_{i^*}$ has size $>\eta$.
This is impossible by the minimality of $\delta$.
Contradiction.
\end{proof}

The conclusions of the theorem are optimal.
A Prikry sequence witnesses this in the countable case and the Magidor iteration of Prikry forcing witnesses this in the uncountable case.

\begin{theorem}
Suppose that $V_1 \supseteq V$ are such that:

$(a)$ $V_1$ and $V$ have the same cardinals and reals,

$(b)$ $\kappa < \lambda$ are infinite cardinals of $V_1$,

$(c)$ there is no splitting $\lan S_{\sigma}:\sigma<\k \ran$ of $\l$ in $V_1$ as in Theorem 2.1$(d).$

Then adding $\kappa-$many Cohen reals over $V_1$ cannot produce $\lambda-$many Cohen reals over $V.$
\end{theorem}

\begin{proof}
Suppose not. Let $\langle r_{\alpha}: \alpha < \lambda  \rangle$ be a sequence of $\lambda-$many Cohen reals over $V$ added after forcing with $\mathbb{C}(\kappa)$ over $V_1$. Let $G$ be $\mathbb{C}(\kappa)-$generic over $V_1$. For each $p\in \mathbb{C}(\kappa)$ set
\begin{center}
$C_{p}=\{\alpha < \lambda: p$ decides $ \lusim{r}_{\alpha}(0)  \}.$
\end{center}
Then by genericity $\lambda = \bigcup_{p \in G} C_{p}.$ Fix an enumeration $\lan p_\xi:\xi<\k  \ran$ of $G,$ and define  a splitting $\lan S_{\sigma}:\sigma<\k \ran$ of $\l$ in $V_1[G]$ by setting $S_\sigma = C_{p_{\sigma}}\setminus \bigcup_{\xi<\sigma}C_{p_{\xi}}.$
By $(a)$ and $(c)$ we can find a countable $I\in V$ and $\sigma<\k$ such that $I\subseteq S_{\sigma}.$ \footnote{In fact, by $(c)$ there exist a countable  $I\in V$ and some $\sigma<\k$ such that $I\cap S_{\sigma}$ is infinite. By $(a)$, $V$ and $V_1$ have the same reals, and hence $I\cap S_{\sigma}\in V.$ So by replacing $I$ with $I\cap S_{\sigma},$ if necessary, we can assume that $I\subseteq S_{\sigma}$.}
 Suppose for simplicity that $\forall \alpha \in S_{\sigma}, p_{\sigma} \vdash \ulcorner \lusim{r}_{\alpha}(0)=0  \urcorner.$ Let $q \in \mathbb{C}(\k)$ be such that
\begin{center}
$q \vdash^{V} \ulcorner I \in V$ is countable and $ \forall \alpha \in I, \lusim{r}_{\alpha}(0)=0 \urcorner.$
\end{center}
Pick $\langle 0, \alpha \rangle \in \omega \times I$ such that $\langle 0, \alpha \rangle \notin supp(q).$ Let $\bar{q}=q \cup \{\langle \langle 0, \alpha \rangle, 1 \rangle \}.$ Then $\bar{q} \in \mathbb{C}(\k), \bar{q} \leq q$ and $\bar{q} \vdash \ulcorner \lusim{r}_{\alpha}(0)=1  \urcorner$, which is a contradiction.
\end{proof}

The following corollary answers a question from [1].
\begin{corollary}
The following are equiconsistent:

$(a)$ There exists a pair $(V_1, V_2), V_1 \subseteq V_2$ of models of set theory with the same cardinals and reals and a cardinal $\kappa$ of cofinality $\omega$ (in $V_2$) such that adding $\kappa-$many Cohen reals over $V_2$ adds more than $\kappa-$many Cohen reals over $V_1.$

$(b)$ There exists a cardinal $\delta$ which is a limit of $ \delta-$many measurable cardinals.
\end{corollary}
\begin{proof}
Assume $(a)$ holds for some pair $(V_1, V_2)$ of models of set theory, $V_1 \subseteq V_2$ which have the same cardinals and reals. If there is a sharp for a strong cardinal, then clearly in $\mathcal{K},$ the core model for a strong cardinal, there is a cardinal $\delta$ which is a limit of $ \delta-$many measurable cardinals \footnote{In fact there are many such cardinals $\delta.$}. So assume there is no sharp for a strong cardinal.  Then by Theorem 3.10 there exists a splitting $\lan S_{\sigma}:\sigma<\k \ran$ of $\k^+$ in $V_2$ such that for every countable set $I\in V_1$ and $\sigma<\k, I\cap S_\sigma$ is finite. Take $S$ to be one of the sets $S_\sigma$ which has size $\k^+.$ So by Theorem 3.9, we get the consistency of $(b)$ \footnote{Note that necessarily case $(b)$ of Theorem 3.9 happens.}.

Conversely if $(b)$ is consistent, then by Corollary 3.8  the consistency of $(a)$ follows \footnote{If $cf(\delta)>\omega,$ then we can find $\delta^* < \delta$ of cofinality $\omega$ which is a limit of $\delta^*-$many measurable cardinals, so that Corollary 3.8 can be applied. To see such a $\delta^*$ exists, define an increasing sequence $\delta_n, n<\omega,$ of cardinals below $\delta,$ so that for any $n,$ there are at least
$\delta_n-$many measurable cardinals below $\delta_{n+1},$ and let $\delta^*=\sup_n \delta_n.$}.
\end{proof}
\section{Below the first fixed point of the $\aleph-$function}

\begin{theorem}
Suppose that $V_1 \supseteq V$ are such that $V_1$ and $V$ have the same cardinals and reals. Suppose $\aleph_{\delta} < $ the first fixed point of the $\aleph-$function, $X \subseteq \aleph_{\delta}, X \in V_{1}$ and $|X| \geq \delta^{+}$ (in $V_1$). Then $X$ has a countable subset which is in $V$.
\end{theorem}
\begin{proof}
By induction on $\delta <$ the first fixed point of the $\aleph-$function.

{\bf Case 1.} $\delta=0.$ Then $X \in V$ by the fact that $V_1$ and $V$ have the same reals.

{\bf Case 2.} $\delta = \delta^{'}+1.$ We have $\delta^{'} < \aleph_{\delta^{'}},$ hence $\delta^{+} < \aleph_{\delta},$ thus we may suppose that $|X| \leq  \aleph_{\delta^{'}}.$ Let $\eta = sup(X) < \aleph_{\delta}.$ Pick $f_{\eta}: \aleph_{\delta^{'}} \leftrightarrow \eta, f_{\eta} \in V.$ Set $Y=f_{\eta}^{-1''}X.$ Then $Y \subseteq \aleph_{\delta^{'}}, \delta^{'} < \aleph_{\delta^{'}}$ and $|Y| \geq \delta^{+} =\delta^{'+}.$ Hence by induction there is a countable set $B \in V$ such that $B \subseteq Y.$ Let $A = f_{\eta}^{''}B.$ Then $A \in V$ is a countable subset of $X$.

{\bf Case 3.} $limit(\delta).$ Let $\langle \delta_{\xi}: \xi < cf\delta \rangle$ be
increasing and cofinal in $\delta.$ Pick $\xi < cf\delta$ such
that $|X \cap \aleph_{\delta_{\xi}}| \geq \delta^{+}.$ By induction there is a countable set $A \in V$ such
 that $A \subseteq X \cap \aleph_{\delta_{\xi}} \subseteq X.$
\end{proof}

The following corollary gives a negative answer to another question from [1].
\begin{corollary}
Suppose $V_{1}, V$ and $\delta$ are as in Theorem 4.1. Then adding $\aleph_{\delta}-$many Cohen reals over $V_1$ cannot  produce $\aleph_{\delta+1}-$many Cohen reals over $V.$
\end{corollary}
\begin{proof}
Towards a contradiction suppose that adding $\aleph_{\delta}-$many Cohen reals over $V_1$  produces $\aleph_{\delta+1}-$many Cohen reals over $V.$ Then by Theorem 3.10, there exists $X \subseteq \aleph_{\delta+1}, X \in V_{1}$ such that $|X|=\aleph_{\delta+1} (\geq \delta^{+})$ and $X$ does not contain any countable subset from $V$ \footnote{In fact, there exists a splitting $\langle S_\sigma: \sigma< \aleph_\delta \rangle$ of $\aleph_{\delta+1}$ in $V_1$, consisting of sets of size $\aleph_{\delta+1}$ such that each $S_\sigma$ has finite intersection with any countable set from $V$. The set $X$ can be chosen to be any of $S_\sigma$'s.}, which is in contradiction with Theorem 4.1.
\end{proof}

\section{At the first fixed point of the $\aleph-$function}
The next theorem shows that Theorem 4.1 does not extend to the first fixed point of the $\aleph-$function.

\begin{theorem}
 Suppose $GCH$ holds and $\kappa$ is the least singular cardinal of cofinality $\omega$ which is a limit of $\kappa-$many measurable cardinals. Then there is a pair $(V[G], V[H])$ of generic extensions of $V$ with $V[G]\subseteq V[H]$ such that:

$(a)$ $V[G]$ and $V[H]$ have the same cardinals and reals,

$(b)$ $\kappa$ is the first fixed point of the $\aleph-$function in $V[G]$ ( and hence in V[H]),

$(c)$ in $V[H]$ there exists a splitting $\langle S_{\sigma}:\sigma<\k  \rangle$ of $\k$ into sets of size $\k$ such that for every

$\hspace{.5cm}$countable $I\in V[G]$ and $\sigma<\k, |I\cap S_{\sigma}|<\aleph_0.$

\end{theorem}
\begin{proof}
We first give a simple observation.
\begin{claim}
Suppose there is $S \subseteq \k$ of size $\k$ in $V[H] \supseteq V[G]$ such that for every countable $A\in V[G], |A \cap S|<\aleph_0.$ Then there is a splitting $\langle S_{\sigma}:\sigma<\k  \rangle$ of $\k$ as in $(c).$
\end{claim}
\begin{proof}
Let $\lan \a_i: i<\k \ran$ be an increasing enumeration of $S.$ We may further suppose that $\a_0=0$, each $\a_i, i>0$ is measurable \footnote{In fact it suffices for each $\a_i$ to be inaccessible in $V$.}  in $V$ and is not a limit point of $S$.\footnote{Let $f\in V$ be such that $f: \k \rightarrow X$ is a bijection, where $X$ is a discrete set of measurable cardinals of $V$ below $\kappa$  of size $\kappa$. Then if $S  \subseteq \k$ satisfies the claim, so does $f[S],$ hence we can suppose all non-zero elements of $S$ are measurable in $V,$ and are not a limit point of $S$.}  Note that for all $i<\k, sup_{j<i}\a_j< \a_i\setminus sup_{j<i}\a_j.$ Now set:

$\hspace{1cm}$ $S_0=S,$

$\hspace{1cm}$ $S_{\sigma}=\{\a_{l}+\sigma: i\leq l <\k \},$ for $0<\sigma\in [sup_{j<i}\a_{j},\a_i ).$

Then $\langle S_{\sigma}:\sigma<\k  \rangle$ is as required (note that for $\sigma>0, S_\sigma \subseteq S+\sigma = \{ \a+\sigma: \a\in S\},$ and  clearly $S+\sigma,$ and hence $S_\sigma$, has finite intersection with countable sets from $V[G]$).
\end{proof}

Thus it is enough to find  a pair $(V[G], V[H])$ of generic extensions of $V$ satisfying $(a)$ and $(b)$ with $V[G] \subseteq V[H]$ such that in $V[H]$ there is $S \subseteq \k$ of size $\k$ composed of inaccessibles,  such that for every countable $A\in V[G], |A\cap S|<\aleph_0.$

 Let $X$ be a discrete set of measurable cardinals below $\kappa$ of size $\kappa,$ and for each $\xi \in X$ fix a normal measure $U_{\xi}$ on $\xi.$
For each $\xi \in X$ we define two forcing notions $\mathbb{P}_{\xi}$ and $\mathbb{Q}_{\xi}$ as follows.
\begin{remark}
 In the following definitions we let  $sup(X \cap \xi)=\omega$ for $\xi=minX.$
\end{remark}

A condition in $\mathbb{P}_{\xi}$ is of the form $p=\langle s_{\xi}, A_{\xi}, f_{\xi} \rangle$ where
\begin{enumerate}
 \item $s_{\xi}  \in [\xi \backslash sup(X \cap \xi)^{+}]^{<2}$,
 \item if $s_{\xi} \neq \emptyset$ then  $ s_{\xi}(0)$ is an inaccessible cardinal, \item $A_{\xi} \in U_{\xi},$ \item $maxs_{\xi} < minA_{\xi},$
 \item $s_{\xi}=\emptyset \Rightarrow f_{\xi} \in Col(sup(X \cap \xi)^{+}, < \xi),$ where $Col(sup(X \cap \xi)^{+}, < \xi)$ is the L\'{e}vy collapse for collapsing all cardinals less than $\xi$ to $sup(X \cap \xi)^{+},$ and making $\xi$  become the successor of $sup(X \cap \xi)^{+},$
 \item $s_{\xi} \neq \emptyset \Rightarrow f_{\xi}=\langle f_{\xi}^{1}, f_{\xi}^{2}   \rangle$ where $f_{\xi}^{1} \in Col(sup(X \cap \xi)^{+}, < s_{\xi}(0))$ and $f_{\xi}^{2} \in Col((s_{\xi}(0))^{+}, < \xi).$
\end{enumerate}
For $p, q \in \mathbb{P}_{\xi}, p=\langle s_{\xi}, A_{\xi}, f_{\xi} \rangle$ and $q=\langle t_{\xi}, B_{\xi}, g_{\xi} \rangle$ we define $p \leq q$ iff
\begin{enumerate}
\item $ s_{\xi}$ end extends $t_{\xi},$ \item $ A_{\xi} \cup (s_{\xi} \backslash t_{\xi}) \subseteq B_{\xi},$
\item $t_{\xi}=s_{\xi}=\emptyset \Rightarrow f_{\xi} \leq g_{\xi},$
\item $t_{\xi}=\emptyset$ and $s_{\xi} \neq \emptyset \Rightarrow sup(ran(g_{\xi})) < s_{\xi}(0)$ and $f_{\xi}^{1} \leq g_{\xi},$
 \item  $t_{\xi} \neq \emptyset \Rightarrow f_{\xi}^{1} \leq g_{\xi}^{1}$ and $f_{\xi}^{2} \leq g_{\xi}^{2}$ (note that in this case we have $s_{\xi}=t_{\xi}$).

\end{enumerate}
We also define $p \leq^{*} q$ ($p$ is a Prikry or a direct extension of $q$) iff
\begin{enumerate}
\item $p \leq q,$ \item $s_{\xi}=t_{\xi}.$
\end{enumerate}
The proof of the following lemma is essentially the same as in the proofs in [2, 5].
\begin{lemma} ($GCH$)
$(a)$ $\mathbb{P}_{\xi}$ satisfies the $\xi^{+}-c.c.$

$(b)$ Suppose $p=\langle s_{\xi}, A_{\xi}, f_{\xi} \rangle \in \mathbb{P}_{\xi}$ and  $l(s_{\xi})=1$ (where $l(s_{\xi})$ is the length of $s_\xi$). Then $\mathbb{P}_{\xi}/ p=\{q \in \mathbb{P}_{\xi}: q \leq p   \}$ satisfies the $\xi-c.c.$

$(c)$ $(\mathbb{P}_{\xi}, \leq, \leq^{*})$ satisfies the Prikry property, i.e., given $p\in \mathbb{P}$ and a sentence $\sigma$ of the forcing language for $(\mathbb{P}, \leq),$ there exists $q\leq^* p$ which decides $\sigma.$

$(d)$ Let $G_{\xi}$ be $\mathbb{P}_{\xi}-$generic over $V$ and let $\langle s_{\xi}(0) \rangle$ be the one element sequence added by $G_{\xi}$. Then in $V[G_{\xi}], GCH$ holds, and the only cardinals which are collapsed are the cardinals in the intervals $(sup(X \cap \xi)^{++}, s_{\xi}(0))$ and $(s_{\xi}(0)^{++}, \xi),$ which are collapsed to $sup(X \cap \xi)^{+}$ and $s_{\xi}(0)^{+}$ respectively.
\end{lemma}
We now define the forcing notion $\mathbb{Q}_{\xi}.$ A condition in $\mathbb{Q}_{\xi}$ is of the form $p=\langle s_{\xi}, A_{\xi}, f_{\xi} \rangle$ where
\begin{enumerate}
 \item $s_{\xi}  \in [\xi \backslash sup(X \cap \xi)^{+}]^{<3},$
 \item if $s_{\xi} \neq \emptyset$ then for all $i < l(s_{\xi}), s_{\xi}(i)$ is an inaccessible cardinal,
 \item $A_{\xi} \in U_{\xi},$ \item $maxs_{\xi} < minA_{\xi},$ \item $s_{\xi}=\emptyset \Rightarrow f_{\xi} \in Col(sup(X \cap \xi)^{+}, < \xi),$
 \item $s_{\xi} \neq \emptyset \Rightarrow f_{\xi}=\langle f_{\xi}^{1}, f_{\xi}^{2}   \rangle$ where, $f_{\xi}^{1} \in Col(sup(X \cap \xi)^{+}, < s_{\xi}(0))$ and $f_{\xi}^{2} \in Col((s_{\xi}(0))^{+}, < \xi).$
\end{enumerate}
For $p, q \in \mathbb{Q}_{\xi}, p=\langle s_{\xi}, A_{\xi}, f_{\xi} \rangle$ and $q=\langle t_{\xi}, B_{\xi}, g_{\xi} \rangle$ we define $p \leq q$ iff
\begin{enumerate}
\item $ s_{\xi}$ end extends $t_{\xi},$ \item $ A_{\xi} \cup (s_{\xi} \backslash t_{\xi}) \subseteq B_{\xi},$
\item $t_{\xi}=s_{\xi}=\emptyset \Rightarrow f_{\xi} \leq g_{\xi},$
\item $t_{\xi}=\emptyset$ and $s_{\xi} \neq \emptyset \Rightarrow sup(ran(g_{\xi})) < s_{\xi}(0)$ and $f_{\xi}^{1} \leq g_{\xi},$
 \item  $t_{\xi} \neq \emptyset$ and $s_{\xi}=t_{\xi} \Rightarrow f_{\xi}^{1} \leq g_{\xi}^{1}$ and $f_{\xi}^{2} \leq g_{\xi}^{2},$
\item  $t_{\xi} \neq \emptyset$ and $s_{\xi} \neq t_{\xi} \Rightarrow sup(ran(g_{\xi}^{2})) < s_{\xi}(1), f_{\xi}^{1} \leq g_{\xi}^{1}$ and $f_{\xi}^{2} \leq g_{\xi}^{2}.$
\end{enumerate}
We also define $p \leq^{*} q$  iff
\begin{enumerate}
\item $p \leq q,$ \item $s_{\xi}=t_{\xi}.$
\end{enumerate}
As above we have the following.
\begin{lemma} ($GCH$)
$(a)$ $\mathbb{Q}_{\xi}$ satisfies the $\xi^{+}-c.c.$

$(b)$ Suppose $p=\langle s_{\xi}, A_{\xi}, f_{\xi} \rangle \in \mathbb{Q}_{\xi}, l(s_{\xi})=2.$ Then $\mathbb{Q}_{\xi}/ p=\{q \in \mathbb{Q}_{\xi}: q \leq p   \}$ satisfies the $\xi-c.c..$

$(c)$ $(\mathbb{Q}_{\xi}, \leq, \leq^{*})$ satisfies the Prikry property.

$(d)$ Let $H_{\xi}$ be $\mathbb{Q}_{\xi}-$generic over $V$ and let $\langle s_{\xi}(0), s_{\xi}(1) \rangle$ be the two element sequence added by $H_{\xi}$. Then in $V[H_{\xi}], GCH$ holds, and the only cardinals which are collapsed are the cardinals in the intervals $(sup(X \cap \xi)^{++}, s_{\xi}(0))$ and $(s_{\xi}(0)^{++}, \xi),$ which are collapsed to $sup(X \cap \xi)^{+}$ and $s_{\xi}(0)^{+}$ respectively.
\end{lemma}
Now let $\mathbb{P}$ be the Magidor iteration of the forcings $\mathbb{P}_{\xi}, \xi \in X,$ and $\mathbb{Q}$ be the Magidor iteration of the forcings $\mathbb{Q}_{\xi}, \xi \in X$. Since the set $X$ is discrete we can view each condition in $\mathbb{P}$ as a sequence $p=\langle \langle s_{\xi}, A_{\xi}, f_{\xi} \rangle: \xi \in X  \rangle$ where for each $\xi \in X, \langle s_{\xi}, A_{\xi}, f_{\xi} \rangle \in \mathbb{P}_{\xi}$ and $supp(p)=\{\xi: s_{\xi} \neq \emptyset \}$ is finite. Similarly each condition in $\mathbb{Q}$ can be viewed as a sequence $p=\langle \langle s_{\xi}, A_{\xi}, f_{\xi} \rangle: \xi \in X  \rangle$ where for each $\xi \in X, \langle s_{\xi}, A_{\xi}, f_{\xi} \rangle \in \mathbb{Q}_{\xi}$ and $supp(p)=\{\xi: s_{\xi} \neq \emptyset \}$ is finite (for more information see [2, 4, 5]).
\begin{notation}
 If $p$ is as above, then we write $p(\xi)$ for $\langle s_{\xi}, A_{\xi}, f_{\xi} \rangle.$
\end{notation}

We also define
\begin{center}
 $\pi: \mathbb{Q} \rightarrow \mathbb{P}$
\end{center}
by
\begin{center}
 $\pi(\langle \langle s_{\xi}, A_{\xi}, f_{\xi} \rangle: \xi \in X  \rangle)=\langle \langle s_{\xi}\upharpoonright 1, A_{\xi}, f_{\xi} \rangle: \xi \in X  \rangle.$
\end{center}
It is clear that $\pi$ is well-defined.
\begin{lemma}
 $\pi$ is a projection, i.e.,

$(a)$ $\pi(1_{\mathbb{Q}})=1_{\mathbb{P}},$

$(b)$ $\pi$ is order preserving,

$(c)$ if $p \in \mathbb{Q}, q \in \mathbb{P}$ and $q \leq \pi(p)$ then there is $r \leq p$ in $\mathbb{Q}$ such that $\pi(r) \leq q.$
\end{lemma}
Now let $H$ be $\mathbb{Q}-$generic over $V$ and let $G=\pi^{''}H$ be the filter generated by $\pi^{''}H.$ Then $G$ is $\mathbb{P}-$generic over $V$.
\begin{lemma}
$(a)$ if $\langle\tau_{\xi}: \xi \in X  \rangle$  and  $\langle \langle \eta_{\xi}^{0}, \eta_{\xi}^{1} \rangle: \xi \in X  \rangle$ are the Prikry sequences added by $G$ and $H$ respectively, then $\tau_{\xi}=\eta_{\xi}^{0}$ for all $\xi \in X.$

$(b)$ The models $V[G]$ and $V[H]$ satisfy the $GCH,$ have the same cardinals and reals, and furthermore the only cardinals of $V$ below $\kappa$ which are preserved are $\{\omega, \omega_{1} \} \cup lim(X) \cup  \{\tau_{\xi}, \tau_{\xi}^{+}, \xi, \xi^{+}: \xi \in X \}.$

$(c)$ $\kappa$ is the first fixed point of the $\aleph-$function in $V[G]$ (and hence in $V[H]$).
\end{lemma}
\begin{proof}
 $(a)$ and $(b)$ follow easily from Lemmas 5.4 and 5.5 and the definition of the projection $\pi.$ Let's prove $(c).$ It is clear that $\k$ is a fixed point of the $\aleph-$function in $V[G]$. On the other hand, by $(b)$, the only cardinals of $V$ below $\k$ which are preserved in $V[G]$ are $\{\omega, \omega_{1} \} \cup lim(X) \cup  \{\tau_{\xi}, \tau_{\xi}^{+}, \xi, \xi^{+}: \xi \in X \},$ and so if $\l < \k$ is a limit cardinal in $V[G]$, then $\l\in lim(X).$  But by our assumption on $\k,$ if $\l\in lim(X),$ then $X\cap \l$ has order type less than $\l,$ and hence $(\{\omega, \omega_{1} \} \cup lim(X) \cup  \{\tau_{\xi}, \tau_{\xi}^{+}, \xi, \xi^{+}: \xi \in X \}) \cap \l$ has order type less than $\aleph_\l.$ Thus $\l<\aleph_\l.$
\end{proof}
Let $\mathbb{Q}/G =\{ p\ \in \mathbb{Q}: \pi(p) \in G \}.$ Then $V[H]$ can be viewed as a generic extension of $V[G]$ by $\mathbb{Q}/G.$
\begin{lemma}
 $\mathbb{Q}/G$ is cone homogenous: given $p$ and $q$ in $\mathbb{Q}/G$ there exist $p^{*} \leq p, q^{*} \leq q$ and an isomorphism $\rho:(\mathbb{Q}/G)/ p^{*} \rightarrow (\mathbb{Q}/G)/ q^{*}.$
\end{lemma}
\begin{proof}
Suppose $p,q \in \mathbb{Q}/G.$  Extend $p$ and $q$ to $p^{*}=\langle \langle s_{\xi}, A_{\xi}, f_{\xi} \rangle: \xi \in X  \rangle$ and $q^{*}=\langle \langle t_{\xi}, B_{\xi}, g_{\xi} \rangle: \xi \in X  \rangle$ respectively so that the following conditions are satisfied:

\begin{enumerate}
\item $supp(p^{*})=supp(q^{*}).$ Call this common support $K$. \item For every
$\xi \in K, l(s_{\xi})=l(t_{\xi})=2.$ Note that then for every $\xi \in K, s_{\xi}(0)=t_{\xi}(0)=\tau_{\xi},$
 $f_{\xi}=\langle f_{\xi}^{1}, f_{\xi}^{2}   \rangle$ and $g_{\xi}=\langle g_{\xi}^{1}, g_{\xi}^{2}   \rangle$ where $f_{\xi}^{1}, g_{\xi}^{1} \in Col(sup(X \cap \xi)^{+}, < \tau_{\xi})$ and $f_{\xi}^{2}, g_{\xi}^{2} \in Col(\tau_{\xi}^{+}, < \xi).$
 \item For every $\xi \in K, A_{\xi}=B_{\xi}.$
\item For every $\xi \in K, dom(f_{\xi}^{1})=dom(g_{\xi}^{1})$ and $dom(f_{\xi}^{2})=dom(g_{\xi}^{2}).$
\item For every $\xi \in K,$ there exists an automorphism $\rho_{\xi}^{1}$  of $Col(sup(X \cap \xi)^{+}, < \tau_{\xi})$ such that $\rho_{\xi}^{1}(f_{\xi}^{1})=g_{\xi}^{1}.$
\item For every $\xi \in K,$ there exists an automorphism $\rho_{\xi}^{2}$  of $Col(\tau_{\xi}^{+}, < \xi)$ such that $\rho_{\xi}^{2}(f_{\xi}^{2})=g_{\xi}^{2}.$
\end{enumerate}
Note that clauses $(5)$ and $(6)$ are possible, as the corresponding forcing notions are homogeneous.

We now define $\rho:(\mathbb{Q}/G)/ p^{*} \rightarrow (\mathbb{Q}/G)/ q^{*}$ as follows. Suppose $r \in \mathbb{Q}/G, r \leq p^{*}.$ Let $r=\langle \langle r_{\xi}, C_{\xi}, h_{\xi} \rangle: \xi \in X  \rangle.$ Then for every $\xi \in K, r_{\xi}=s_{\xi},$ and $h_{\xi}=\langle h_{\xi}^{1}, h_{\xi}^{2}   \rangle$
where  $h_{\xi}^{1} \in Col(sup(X \cap \xi)^{+}, < \tau_{\xi})$ and $h_{\xi}^{2} \in Col(\tau_{\xi}^{+}, < \xi).$ Let
\begin{center}
 $\rho(r)=\langle \langle t_{\xi}, C_{\xi}, \langle \rho_{\xi}^{1}(h_{\xi}^{1}), \rho_{\xi}^{2}(h_{\xi}^{2}) \rangle  \rangle: \xi \in K \rangle ^{\frown} \langle \langle  r_{\xi}, C_{\xi}, h_{\xi}    \rangle: \xi \in X \setminus K \rangle.$
\end{center}
It is easily seen that $\rho$ is an isomorphism from $(\mathbb{Q}/G)/ p^{*}$ to $(\mathbb{Q}/G)/ q^{*}.$

\end{proof}

The following lemma completes the proof.
\begin{lemma}
 Let $S=\{\eta_{\xi}^{1}:\xi \in X \}.$ Then $S$ is a subset of $\kappa$ of size $\kappa$ and $|A \cap S|< \aleph_{0}$ for every countable set $A \in V[G].$
\end{lemma}
\begin{remark}
 $(a)$ Since $V[G]$ and $V[H]$ have the same reals, it suffices to prove the lemma for $A \subseteq S, A\in V[G].$ In fact suppose that the lemma is true for all countable  $A \subseteq S, A\in V[G].$ If the lemma fails, then for some countable set $B\in V[G], |B\cap S|=\aleph_0.$ Let $g: \omega \rightarrow B$ be a bijection in $V[G]$. Then $g^{-1}[B\cap S]$ is a subset of $\omega$ which is in $V[H],$ and hence in $V[G].$ Thus $B\cap S\in V[G].$ Hence we find a countable subset  $A \subseteq S$ in $V[G],$ namely $B\cap S,$ for which the lemma fails, which is in contradiction with our initial assumption.

$(b)$ In what follows we say $A$ codes $\xi$ (for $\xi \in X$), if $\eta_{\xi}^{1} \in A.$
\end{remark}

\begin{proof}

 Let $\lusim{S}$ be a $\mathbb{Q}/G-$name for $S.$  Also let $p_{0} \in H \cap \mathbb{Q}/G$ be such that  $p_{o} \vdash_{\mathbb{Q}/G}^{V[G]} \ulcorner \check{A} \subseteq \lusim{S} $ is countable$\urcorner.$
\begin{claim}
 For every $p \in \mathbb{Q}/G$ and every $\xi \in X \setminus supp(p)$ there is  $q \leq p$ in $\mathbb{Q}/G$ such that $\xi \in supp(q)$ and if $q(\xi)=\langle s_{\xi}, A_{\xi}, f_{\xi}  \rangle,$ then $l(s_{\xi})=2$ and $q \vdash_{\mathbb{Q}/G}^{V[G]} \ulcorner s_{\xi}(1) \notin \check{A} \urcorner.$
\end{claim}
\begin{proof}
Let $p$ and $\xi$ be as in the claim. First pick $\langle\langle\langle t_{\xi}(0)\rangle, A_{\xi}, f_{\xi} \rangle\rangle \in G,$ and then let $q=p ^{\frown} \langle \langle s_{\xi}, A_{\xi}, f_{\xi} \rangle  \rangle,$ where $s_{\xi}(0)=t_{\xi}(0)=\tau_{\xi}$ , $s_{\xi}(1)< \xi$ is large enough so that $s_{\xi}(1) \notin A,$ $sup(ran(f_{\xi}^{2}))<s_{\xi}(1)$ and $s_{\xi}(1)$ is inaccessible. Then $\pi( \langle \langle s_{\xi}, A_{\xi}, f_{\xi} \rangle  \rangle)=\langle\langle\langle t_{\xi}(0)\rangle, A_{\xi}, f_{\xi} \rangle\rangle \in G.$ On the other hand $\pi(p) \in G.$ Let $r \in G, r \leq \pi(p), \langle\langle\langle t_{\xi}(0)\rangle, A_{\xi}, f_{\xi} \rangle\rangle.$ Then $r \leq \pi(q),$ hence $\pi(q) \in G.$ This implies that $q \in \mathbb{Q}/G.$ Clearly $q$ satisfies the requirements of the Claim.
\end{proof}
It follows that the set
\begin{center}
 $D=\{p \in \mathbb{Q}/G: \forall \xi \in X\setminus supp(p)$ there exists $q \leq p$ as in the above Claim$   \}$
\end{center}
is dense open in $\mathbb{Q}/G.$ Let $p \in H \cap D.$ We can assume that $p \leq p_{0}.$ We show that $p \vdash_{\mathbb{Q}/G}^{V[G]} \ulcorner $if $\check{A}$ codes $\xi$ then $\xi \in supp(p)    \urcorner.$ To see this suppose that $\xi \in X \setminus supp(p).$ Thus by Claim 5.12 we can find $q \leq p$ in $\mathbb{Q}/G$ such that $\xi \in supp(q)$ and if $q(\xi)=\langle s_{\xi}, A_{\xi}, f_{\xi}  \rangle,$ then $l(s_{\xi})=2$ and $q \vdash_{\mathbb{Q}/G}^{V[G]} \ulcorner s_{\xi}(1) \notin \check{A} \urcorner.$  It then follows that $\sim p \vdash_{\mathbb{Q}/G}^{V[G]} \ulcorner s_{\xi}(1) \in \check{A} \urcorner.$ But then by the cone homogeneity of $\mathbb{Q}/G$  we have $p \vdash_{\mathbb{Q}/G}^{V[G]} \ulcorner s_{\xi}(1) \notin \check{A} \urcorner$ \footnote{If not, then for some $p'\leq p, p'\Vdash_{\mathbb{Q}/G}^{V[G]} \ulcorner s_{\xi}(1) \in \check{A} \urcorner.$ By cone homogeneity of $\mathbb{Q}/G$ we can find $q^*\leq q, p^*\leq p'$ and an isomorphism $\rho:(\mathbb{Q}/G)/ p^{*} \rightarrow (\mathbb{Q}/G)/ q^{*}.$ But then by standard forcing arguments and the fact that $q^* \vdash_{\mathbb{Q}/G}^{V[G]} \ulcorner s_{\xi}(1) \notin \check{A} \urcorner,$ we can conclude that $p^* \vdash_{\mathbb{Q}/G}^{V[G]} \ulcorner s_{\xi}(1) \notin \check{A} \urcorner,$ which is impossible, as $p^*\leq p'$ and $p'\Vdash_{\mathbb{Q}/G}^{V[G]} \ulcorner s_{\xi}(1) \in \check{A} \urcorner$.}. Hence $p \vdash_{\mathbb{Q}/G}^{V[G]} \ulcorner \check{A}$ does not code $\xi  \urcorner.$  This means that $p \vdash_{\mathbb{Q}/G}^{V[G]} \ulcorner \check{A} \subseteq \{s_{\xi}(1):\xi \in supp(p) \}=\{\eta_{\xi}^{1}:\xi \in supp(p)   \}  \urcorner.$ Lemma 5.10 follows by noting that $p \in H$ and since the Magidor iteration is used, the support of any condition is finite.
\end{proof}
Theorem 5.1 follows.
\end{proof}

The following theorem can be proved by combining the methods of the proofs of Theorems 3.1 and 5.1.

\begin{theorem}
 Suppose $GCH$ holds and $\kappa$ is the least singular cardinal of cofinality $\omega$ which is
 a limit of $\kappa-$many measurable cardinals. Also let   $V[G]$ and $V[H]$ be the models
 constructed in the proof of Theorem 5.1. Then
 there is a cardinal preserving, not adding a
 real generic extension $V[H][K]$ of $V[H]$ such that in $V[H][K]$ there exists a
 splitting $\lan S_{\sigma}:\sigma<\k \ran$ of $\k^+$ into sets of size $\k^+$ such that for every countable set $I\in V[G]$ and $\sigma<\k, |I\cap S_{\sigma}|<\aleph_0.$
\end{theorem}
\begin{proof}
Work over $V[H]$ and force the splitting $\lan S_{\sigma}:\sigma<\k \ran$ as in the proof of Theorem 3.1,
with $V, V[G]$ used there  replaced by $V[G], V[H]$ here respectively. The  role of the sequence $\bigcup_{\xi\in X}x_\xi$ in the proof of Theorem 3.1 is now played by the sequence $S=\{ \eta_\xi^1: \xi\in X\}$.
\end{proof}

\begin{corollary}
 Suppose $GCH$ holds and there exists a cardinal $\kappa$ which is of cofinality $\omega$ and is a limit of $\kappa-$many
 measurable cardinals. Then there is pair $(V_1, V_2)$ of models of $ZFC, V_1 \subseteq V_2$ such that:

$(a)$ $V_1$ and $V_2$ have the same cardinals and reals.

$(b)$ $\kappa$ is the first fixed
point of the $\aleph-$function in $V_1$ (and hence in $V_2$).

$(c)$ Adding $\kappa-$many Cohen reals over $V_2$ adds
$\kappa^{+}-$many Cohen reals over $V_1.$
\end{corollary}
\begin{proof}
Let $V_1=V[G]$ and $V_2=V[H][K],$ where $V[G], V[H][K]$ are as in Theorem 5.13. The result follows using Remark 2.2 and Theorem 5.13.
\end{proof}

\noindent{\large\bf Acknowledgement}

The authors would like to thank the referee of the paper for his/her many helpful suggestions.

Moti Gitik,
School of Mathematical Sciences, Tel Aviv University, Tel Aviv, Israel.

E-mail address: gitik@post.tau.ac.il

http://www.math.tau.ac.il/~gitik/

Mohammad Golshani,
School of Mathematics, Institute for Research in Fundamental Sciences (IPM), P.O. Box:
19395-5746, Tehran-Iran.

E-mail address: golshani.m@gmail.com

http://math.ipm.ac.ir/~golshani/

\end{document}